\documentclass[psamsfonts,11pt]
{amsart}
\usepackage{amssymb}
\usepackage{mathrsfs}

\input xy  \xyoption{all}

\usepackage[left=1.0in,right=1.0in,top=1.0in,bottom=1.0in]{geometry}

\newtheorem{deff}{Definition}[section]
\newtheorem{lemma}[deff]{Lemma}
\newtheorem{theorem}[deff]{Theorem}
\newtheorem{corollary}[deff]{Corollary}

\newtheorem{fact}[deff]{Fact}
\newtheorem{em-example}[deff]{Example}
\newtheorem{em-def}[deff]{Definition}        
\newtheorem{em-remark}[deff]{Remark}         
\newtheorem{em-question}[deff]{Question}

\newenvironment{example}{\begin{em-example} \em }{ \end{em-example}}
\newenvironment{definition}{\begin{em-def} \em  }{ \end{em-def}}
\newenvironment{remark}{\begin{em-remark} \em }{\end{em-remark}}
\newenvironment{question}{\begin{em-question}\em }{\end{em-question}}

\newcommand{\een}{\end{enumerate}}
\newcommand{\bit}{\begin{itemize}}

\newcommand{\eit}{\end{itemize}}

\newcommand\DEQ{\hfill $\bigtriangleup$ \medskip}

\def\nbd{neighbourhood{}}

\def\hull#1{\langle#1\rangle}

\def\ker{\mathop{\rm ker}}

\def\:{\nobreak \hskip .1111em\mathpunct {}\nonscript \mkern
-\thinmuskip {:}\hskip .3333emplus.0555em\relax}

\catcode`\@=12
\def\T{{\mathbb T}}

\def\Z{{\mathbb Z}}
\def\N{{\mathbb N}}

\def\R{{\mathbb R}}
\def\Q{{\mathbb Q}}
\def\P{{\mathbb P}}

\def\Prm{\P}

\def\cont{\mathfrak c}

\newcommand{\cc}{countably\ compact}
\newcommand{\mi}{minimal}

\newcommand{\cag}{compact abelian group}

\def\ct{{countable tightness}}
\def\aa{$\chi$-abelian}

\def\cont{\mathfrak{c}}

\def\seq{\mathfrak{s}}
\def\hull#1{\langle{#1}\rangle}

\def\Min{\mathbf{Min}}

\def\grp#1{\langle{#1}\rangle}

\title[Cardinal invariants of locally minimal abelian groups
]{Cardinal invariants and convergence properties of locally minimal groups} 

\author[D. Dikranjan]{Dikran Dikranjan}
\address[Dikran Dikranjan]{Dipartimento di Matematica e Informatica, Universit\`a di Udine
\\
 via delle Scienze, 206 - 33100 Udine, Italy}
\email{dikran.dikranjan@dimi.uniud.it}

\author[D. Shakhmatov]{Dmitri Shakhmatov}
\address[Dmitri Shakhmatov]{Graduate School of Science and Engineering,
Division of Mathematics, Physics and Earth Sciences\\
Ehime University, Matsuyama 790-8577, Japan}
\email{dmitri.shakhmatov@ehime-u.ac.jp}

\begin{document}
\begin{abstract} 
If $G$ is a locally essential subgroup of a compact abelian group $K$, then:
\begin{itemize}
\item[(i)] $t(G)=w(G)=w(K)$, where $t(G)$ is the tightness of $G$;
\item[(ii)] if $G$ is radial, then $K$ must be metrizable;
\item[(iii)] $G$ contains a super-sequence $S$ converging to $0$ such that 
$|S|=w(G)=w(K)$.
\end{itemize}

Items (i)--(iii) hold when $G$ is a dense locally minimal subgroup of $K$. We show that locally minimal locally precompact abelian groups of countable tightness are metrizable. In particular, a minimal abelian group of countable tightness is metrizable. This answers a question of O.~Okunev posed in 2007.

For every uncountable cardinal $\kappa$, we construct a  Fr\' echet-Urysohn minimal group $G$ of character $\kappa$ such that the connected component of $G$ is an open normal $\omega$-bounded subgroup (thus, $G$ is locally precompact).  
We also build a minimal nilpotent group of nilpotency class 2 without non-trivial convergent  sequences
having an open normal countably compact subgroup.
\end{abstract}


\thanks{The first author gratefully acknowledges the FY2013 Long-term visitor grant~L13710 by the Japan Society for the Promotion of Science (JSPS)
 and the grant PRID (Topological, Categorical and Dynamical Methods in Algebra, DMIF) at the University of Udine.}

\thanks{The second author was partially supported by the Grant-in-Aid for Scientific Research (C) No.~22540089 by the Japan Society for the Promotion of Science (JSPS)}

\dedicatory{Dedicated to Misha Megrelishvili on the occasion of his 60th anniversary}

\maketitle

{\it All topological groups are assumed to be Hausdorff.\/}

\section{Preliminaries}

A topological group is {\em (locally) precompact\/} if it is topologically isomorphic to a subgroup of some (locally) compact group, or equivalently, if its completion with respect to the left uniformity is (locally) compact.

Symbols $t(X)$, $\chi(X)$, $nw(X)$, $w(X)$ denote tightness, character, network weight and weight of a space $X$, respectively.
{\em All cardinal invariants are assumed to take infinite values.\/}

\begin{definition}
\label{sigma:product}
For a non-trivial group $K$ and a cardinal $\kappa$, we define
$$
\Sigma_\kappa(K)=\{f\in K^\kappa: |\{\alpha<\kappa: f(\alpha)\not=1\}|\le\omega\}
$$
to be the $\Sigma$-product of $\kappa$-many copies of $K$.
\end{definition}

Recall that a space $X$ is said to be {\em $\omega$-bounded\/} if the closure in $X$ of every countable subset of $X$ is compact.

\begin{fact}
\label{sigma-product:fact}
Let $K$ be a non-trivial compact metric group and $\kappa$ be an uncountable cardinal. Then $G=\Sigma_\kappa(K)$ is an $\omega$-bounded, Fr\' echet-Urysohn group such that $\chi(G)=w(G)=\kappa$. Moreover, the following properties pass from $K$ to $G$:
\begin{itemize}
\item[(i)] connectedness;
\item[(ii)] zero-dimensionality.
\end{itemize}
\end{fact}

For topological groups $G$ and $H$, the symbol $G\cong H$ means that $G$ and $H$ are topologically isomorphic; that is, there exists an isomorphism between $G$ and $H$ which is also a homeomorphism.

\section{Background on (locally) minimal groups}
\label{sec:2}

\begin{definition}
A Hausdorff group topology $\mathscr{T}$ on  a group $G$ is called:
\begin{itemize}
\item[(i)]
\emph{minimal\/} if every Hausdorff group topology $\mathscr{T}'$ on $G$ such that $\mathscr{T}'\subseteq \mathscr{T}$ satisfies $\mathscr{T}' = \mathscr{T}$, 
\item[(ii)]
{\em locally minimal\/}   if there exists $U\in\mathscr{T}$ such that $1\in U$ and 
$\mathscr{T}' = \mathscr{T}$ holds for every Hausdorff group topology $\mathscr{T}'\subseteq \mathscr{T}$ on $G$ for which 
$U$ is a $\mathscr{T}'$-neighborhood of $1$. (Here $1$ denotes the identity element of $G$.)

\end{itemize}
 \end{definition}

Obviously,
\begin{equation}
\label{minimal:is:locally:minimal}
\mbox{compact}\to\mbox{minimal}\to\mbox{locally minimal}.
\end{equation}

The notion of a minimal group was introduced independently by Choquet (see Do\" \i tchinov \cite{Doi}) and Stephenson \cite{St} almost 
fifty years ago as an important generalization of compactness for topological groups. 

Even though compact groups are minimal, local compactness need not imply minimality as the following theorem, due to Stephenson \cite{St}, shows: 

\begin{theorem}
\label{Steph}
A locally compact abelian group is minimal if and only if it is compact.
\end{theorem}

Motivated by the fact that local compactness need not imply minimality, Morris and Pestov \cite{MP}  introduced the class of locally minimal groups (see also Banakh \cite{Ban}).  The importance of this notion stems from the following fact pointed out in \cite{MP}:

\begin{fact}\label{Last:Fact} Both locally compact groups and subgroups of Banach-Lie groups are locally minimal.
\end{fact}
  
The following fundamental result is due to Prodanov and Stoyanov (see \cite{DPS}):

\begin{theorem}
\label{PS}
Every minimal abelian group is precompact.
\end{theorem}

\begin{definition} 
\label{def:essential:local:essential}
A subgroup $G$ of a topological group $H$ is called: 
\begin{itemize}
\item[(i)] {\em essential\/} (in $H$) if $G\cap N=\{1\}$ for a closed normal subgroup $N$ of $H$ implies $N=\{1\}$ (see \cite{B,P,St});
\item[(ii)] {\em locally essential\/} (in $H$) if there exists a neighborhood $V$ of $1$ in $G$ such that  $H\cap N=\{1\}$ implies $N=\{1\}$ for all closed normal subgroups $N$ of $G$ contained in $V$ (see \cite{LocMin}). 
\end{itemize}
\end{definition}

The next criterion is due to Banaschewski, Prodanov and Stephenson \cite{B,P,St}, while its local version (given in parenthesis) is due to Au\ss enhofer,  Chasco, Dikranjan and Dom\'{\i}nguez \cite{LocMin}.

\begin{theorem}\label{crit} A dense subgroup $G$ of a topological group $H$ is (locally) minimal if and only if  $H$ is (locally) minimal and $G$ is (locally) essential in $H$.
\end{theorem}

Using the fact that (locally) compact groups are (locally) minimal, one obtains the following immediate corollary: 

\begin{corollary}\label{cororllary:crit}
 A (locally) precompact group is (locally) minimal if and only if it is a (locally) essential subgroup of its completion.  
\end{corollary}
 
The following diagram lists all implications that are necessary for our discussion {\em in the class of abelian groups\/}. The restriction to the abelian  case was chosen in order to simplify the diagram. The only change that must be carried out in the non-abelian case is to replace 
``minimal''
 by ``minimal+precompact'', which makes the arrow (PS) vacuous. The rest remains unchanged. 

\vskip0.15in
$\phantom{MM}
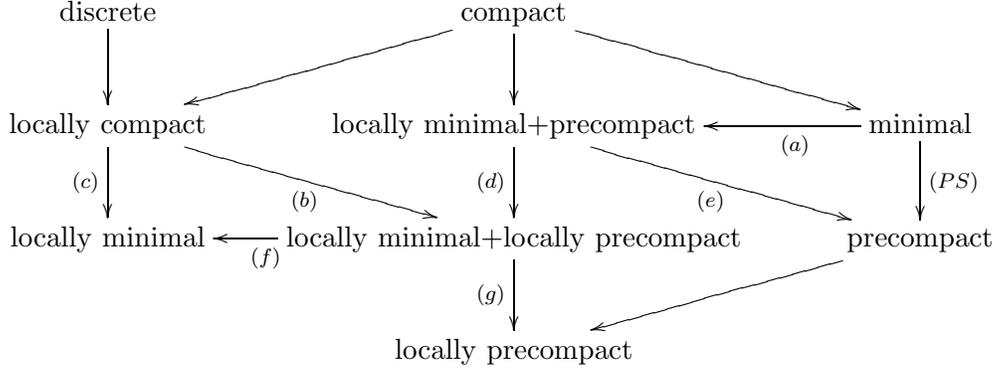
{\xymatrix@!0@C5.4cm@R=1.5cm{
   \mbox{discrete} \ar@{->}[d]
   &  \mbox{compact}\ar@{->}[d]\ar@{->}[dl] \ar@{->}[dr]& & & \\
  \mbox{locally compact} \ar@{->}[dr]_{(b)}\ar@{->}[d]_{(c)}& \mbox{locally minimal+precompact}\ar@{->}[d]_{(d)}\ar@{->}[dr]_{(e)}
  & \mbox{minimal} \ar@{->}[l]^{\;\;\;\;\;\;\;\;\;\;\;\;\;\;\;\;\;\;\;\;\;\;\;\;\;(a)}\ar@{->}[d]^{(PS)}& \\
   \mbox{locally minimal}
  &\mbox{locally minimal+locally precompact}  \ar@{->}[l]^{\!\!\!\!\!\!\!\!\!\!\!\!\!\!\!\!\!\!\!\!\!\!\!\!(f)} \ar@{->}[d]_{(g)}
 &  \mbox{precompact}\ar@{->}[dl]
  & \\
     &\mbox{locally precompact} \\
    } }$
\vskip0.1in

The implication (PS) and the non-trivial part of the implication (a) follow from  Theorem \ref{PS}. The remaining implications (b)--(g) are trivial.

Let us give a list of examples proving that the following implications are not reversible: 

\begin{itemize}
 \item[(a)] Any infinite  cyclic  subgroup $G$ of $\T$ will do.
Indeed, every infinite subgroup of $\T$ is dense and locally essential in $\T$, so
 locally minimal by Theorem \ref{crit}.  
Clearly, no cyclic subgroup of $\T$ is essential in $\T$, so $G$ is not minimal by Theorem \ref{crit}.  
 \item[(b)] Any minimal non-compact group $G$ will do, as such a group cannot be locally compact by Theorem \ref{Steph}, while it is  locally minimal       and locally precompact by the the composition of the  implications (a) and (d). 
 \item[(c)] Follows from the example in item (b) due to the implication (f).
 \item[(d)] The locally compact group $\R$ is not precompact, yet it is locally minimal and locally precompact by the implication (b).
  \item[(e)] The non-invertibility of this implication follows from \cite[Proposition 5.17]{LocMin}. To obtain another example,  one can simply take any dense cyclic subgroup $C$ of $\T^\omega$. Clearly, $C$ is precompact. In Remark~\ref{remark:(e)} below, we shall show that 
$C$ is not locally minimal. 
  \item[(f)] Let $E$ be any infinite dimensional Banach space. Then it is locally minimal by Fact~\ref{Last:Fact}. 
  Being complete and non-locally compact, $E$ cannot be locally precompact. 
 \item[(g)] Take the group $C$ from item (e). 
 \item[(PS)] The group from item (a) will do, due to the implication (e). 
\end{itemize}

The irreversibility of the remaining implications in the diagram is trivial.  

An infinite discrete group is locally minimal and locally precompact, yet it is not precompact. The precompact group $C$ witnessing the irreversibility of (e) is not locally minimal. This shows that the properties ``locally minimal + locally precompact'' and ``precompact'' are independent of each other.

\section{Motivation}

For every Tychonoff space one has the following implications:

\medskip 

\begin{center}
countable weight $\buildrel{(\alpha)}\over\Longrightarrow$ metrizable $\buildrel{(\beta)}\over\Longrightarrow$  Fr\' echet-Urysohn $\buildrel{(\gamma)}\over\Longrightarrow$  sequential $\buildrel{(\delta)}\over\Longrightarrow$  countable tightness. 
\end{center}

\medskip 

For compact spaces the implication ($\alpha$) is invertible, while ($\beta$) and ($\gamma$) are not invertible. It was an open question for some time whether the implication ($\delta$) is invertible, i.e., whether a compact space of countable tightness is sequential. Now it is known that the answer to this question is independent of (the axioms of) ZFC \cite{Ost, Balogh, Dow}.

For compact topological groups all four implications ($\alpha$)--($\delta$) are invertible, as was shown by Ismail \cite{I}:
\begin{theorem}
\label{Ismail:theorem}
\begin{itemize}
\item[(i)] $t(K) = w(K)$ for every compact group $K$.
\item[(ii)] $t(G) = \chi(G)$ for every locally compact group $G$.
\end{itemize}
\end{theorem}
Item (ii) of this theorem shows that implications ($\beta$)--($\delta$) are invertible even for locally compact groups. Clearly, an uncountable discrete group witnesses the fact that the implication $(\alpha$) is not invertible for locally compact groups.

From Theorem \ref{Ismail:theorem}(i) one easily concludes that both implications ($\gamma$) and ($\delta$) are invertible for $\omega$-bounded groups, while the next example shows that the implication ($\beta$) is not invertible even in the abelian case.

\begin{example}\label{example1}  Let $\T$ be the circle group and $\kappa$ be an uncountable cardinal. By Fact \ref{sigma-product:fact}, the group $G=\Sigma_\kappa(\T)$
from Definition \ref{sigma:product} is $\omega$-bounded, Fr\' echet-Urysohn, abelian and satisfies $\chi(G)=w(G)=\kappa$; in particular, $G$ is not metrizable.
The group $G$ is connected by item (i) of Fact \ref{sigma-product:fact}. To obtain  a zero-dimensional  example with similar properties, one can take $G=\Sigma_\kappa(F)$, where $F$ is any finite abelian group; this follows from Fact \ref{sigma-product:fact}(ii).
\end{example}

Even the implication ($\delta$) need not be invertible for countably compact abelian groups, as numerous consistent examples of hereditarily separable countably compact abelian groups without non-trivial convergent sequences demonstrate.  It is consistent with ZFC that the implication $(\gamma$) is invertible for countably compact groups  \cite{Shibakov}.

The following result shows that locally minimal groups have at least one property in  common   with compact spaces:
\begin{theorem}
\label{weight:net-weight} $nw(G)=w(G)$ for every locally minimal group $G$.
\end{theorem}

For minimal groups, Theorem \ref{weight:net-weight} is a classical result of Arhangel'ski\u\i\ \cite{Arh}, while its locally minimal version is more recent \cite[Theorem 2.8]{LocMin}.

Locally minimal abelian groups of countable pseudocharacter are metrizable \cite{LocMin}.  As was proved independently  by Guran \cite{Gu}, Pestov \cite{Pe} and Shakhmatov \cite{Sh1}, non-abelian minimal groups of countable pseudocharacter need not  be metrizable. In fact, the gap between the character and the pseudocharacter of a minimal group can be arbitrarily large \cite{Sh1}. 

The above discussion suggests that minimality may be an appropriate  compactness-like property to consider in order to invert implications ($\alpha$)--($\delta$), but it becomes also  clear that perhaps some sort of commutativity should be assumed. Indeed, we prove that all four implications ($\alpha$)--($\delta$) are simultaneously invertible for minimal abelian groups and that most of these implications cannot be inverted for general minimal groups. Some of our results hold for locally minimal groups as well.

\section{Coincidence of tightness and character in locally minimal abelian groups}
\label{sec:4}

For every topological space $X$ one has $t(X)\leq \chi(X)$. In this section we discuss our results concerning the coincidence of these two invariants for locally minimal abelian groups. 

The next theorem is one of our main results: 

\begin{theorem}
\label{corollary:A*}
\label{theorem:A*}
If $G$ is a locally essential subgroup of a locally \cag\ $K$, then $\chi(K)\leq t(G)$. Moreover, if $K$ is compact, then $w(K)\leq t(G)$.
\end{theorem}

The proof of this theorem will be given in \S \ref{Proofs}. 

\begin{corollary} A locally compact abelian group containing a locally essential subgroup of countable tightness is metrizable.
\end{corollary}

\begin{corollary}\label{corollary:t=chi} 
Let $G$ be an abelian topological group.
\begin{itemize}
  \item[(i)] If $G$ is locally minimal and locally precompact, then $t(G)=\chi(G)$. 
  \item[(ii)] If $G$ is locally minimal and precompact, then $w(G)=t(G)$.
  \item[(iii)] If $G$ is minimal, then $w(G)=t(G)$.
\end{itemize}
\end{corollary}

\begin{proof} To prove (i) and (ii), assume that  $G$ is locally  \mi \ and (locally) precompact. Then the completion $K$ of $G$ is a (locally) compact abelian group. Moreover, by Theorem \ref{crit}, $G$ is locally essential in $K$. Now  the main assertion of Theorem \ref{corollary:A*} yields $\chi(G)\le \chi(K)\leq t(G)$ in case (i), and  the second assertion of Theorem \ref{corollary:A*} yields $w(G)\le w(K)\leq t(G)$ in case (ii). The converse inequalities are trivial.

(iii) follows from \eqref{minimal:is:locally:minimal}, Theorem \ref{PS} and (ii).
\end{proof}

 Since $t(X)\le nw(X)$ for every space $X$ by \cite[Exercise 3.12.7(e),(f)]{Eng}, Corollary \ref{corollary:t=chi}(iii)  strengthens Theorem \ref{weight:net-weight} for abelian groups.

We do not know whether the hypothesis ``locally precompact'' can be omitted in Corollary \ref{corollary:t=chi}(i); see Question \ref{question:0}. 
    
 \begin{corollary}\label{corollary:ct> metrizable}
\begin{itemize}
  \item[(i)] A locally minimal, locally precompact abelian group of countable tightness is metrizable.
  \item[(ii)] Every \mi {} abelian group of \ct \ is metrizable.
\end{itemize}

In particular, 
\begin{itemize}
  \item[(iii)] sequential (or Fr\' echet-Urysohn) locally minimal and locally precompact abelian groups are metrizable.
  \item[(iv)] sequential (or Fr\' echet-Urysohn) \mi {} abelian groups are metrizable.
\end{itemize}
\end{corollary}

Item (ii) of the above corollary provides a positive answer to a question of Oleg Okunev set in 2007. It is worth emphasizing that even item (iv) of Corollary \ref{corollary:ct> metrizable} is completely new.

The reason to pay special attention to items (iii) and (iv) of the above corollary comes from the fact that they treat the invertibility of the implication 
($\beta$). This issue has been faced by the so called Malykhin's problem: are countable Fr\' echet-Urysohn groups metrizable? Here countable can be replaced by separable. Many  papers have been dedicated to Malykhin's problem,  see \cite{MT,Sh2,Sh3} for more detail.  Recently, a model of ZFC in which every countable (and hence every separable) Fr\' echet-Urysohn group is metrizable was constructed in  \cite{HRG}. Here we see that replacing  separability (countability) by minimality allows us to invert the implications  ($\beta$) and ($\gamma$) in ZFC.

Recall that a topological space $X$ is called {\em radial\/} provided that, for every point $x$ in the closure of a subset $A$ of $X$, one can find a well-ordered set $(D,\le)$ and an indexed set $(a_d)_{d\in D}$ of points of $A$ such that the net $(a_d)_{d\in D}$ converges to $x$; the latter means that for every open neighbourhood $U$ of $x$ there exists $c\in D$ (depending on $U$) such that $a_d\in U$ whenever $d\in D$ and $c\le d$. Since every well-ordered set is order isomorphic to some ordinal, without loss of generality, one can assume the well-ordered set $(D,\le)$ in this definition to be an ordinal. 

Clealry, radiality of a space is a generalization of its Fr\' echet-Urysohn property.

\begin{theorem}
\label{theorem:B*} A \cag \ containing a  locally essential radial subgroup is metrizable.
\end{theorem}

The proof of this theorem will be given in \S \ref{Proofs}. 

Theorem \ref{theorem:B*} is a generalization of result of Arhangel'skii stating that compact radial groups are metrizable. (This result follows from the dyadicity of all compact groups and \cite[Corollary 2]{Arh-rad}.)

Our next corollary generalizes the Fr\' echet-Urysohn version of Corollary \ref{corollary:ct> metrizable}(iv).

\begin{corollary}\label{radial}
Every radial locally \mi {}, locally precompact abelian group 
is metrizable. In particular, every radial   \mi {} abelian group is metrizable.
 \end{corollary} 

Example \ref{example1} shows that (local) minimality cannot be traded for some other strong compactness properties in all results in this section.

\section{How much commutativity is necessary? The class of $\chi$-abelian groups}

The reader may wonder whether commutativity is necessary for the proof of Theorem \ref{corollary:A*} and its corollaries in Section \ref{sec:4}.
In order to answer this question, we introduce a new class of topological groups.

\begin{definition}
\label{def:chi-abelian}
Call a topological group $G$ {\em {character-abelian}} (or, briefly, {\em $\chi$-abelian}), if $\chi(Z(G)) = \chi(G)$, where $Z(G)$ denotes the center of the group $G$. 
\end{definition}

Clearly, abelian groups are $\chi$-abelian. Since  all our cardinal functions are assumed to take infinite values, it follows from
Definition \ref{def:chi-abelian} that all metrizable groups are  \aa. 

Our next theorem extends Corollary \ref{corollary:t=chi}(i) from abelian groups to $\chi$-abelian groups.

\begin{theorem}
\label{theorem:LAST2}
 If $G$ is a locally minimal, locally precompact group, then $t(Z(G)) = \chi(Z(G))$. Moreover, if $G$ is also \aa, then$t(G)=\chi(G)$.
\end{theorem}

 \begin{proof} Since 
 $G$ is a locally minimal group, its subgroup $Z(G)$ is locally minimal, as local minimality is preserved by taking closed central subgroups \cite{LocMin}.  Since $G$ is locally precompact, so is its subgroup $Z(G)$.   Since $Z(G)$ is abelian,  it satisfies the equality $t(Z(G)) = \chi(Z(G))$ by Corollary \ref{corollary:t=chi}(i). If $G$ is also $\chi$-abelian, then $\chi(Z(G))=\chi(G)$ by Definition \ref{def:chi-abelian}.
Therefore, $\chi(G)=t(Z(G))\le t(G)$. Since the inverse inequality $t(G)\le \chi(G)$ always holds, this proves the equality $t(G)=\chi(G)$. 
 \end{proof}
  
Our next corollary shows that the class of $\chi$-abelian groups is precisely the (largest) class of topological groups where all minimal groups
 of countable tightness are metrizable.
 
\begin{corollary}
\label{corollary:LAST} A locally minimal, locally precompact group $G$ of countable tightness is metrizable if and only if $G$ is \aa.
\end{corollary}
 
 \begin{proof} 
Let $G$ be a locally minimal, locally precompact group satisfying $t(G)=\omega$. If $G$ is \aa,  
$\chi(G)=t(G)=\omega$ by Theorem \ref{theorem:LAST2} and our assumption, so $G$ is metrizable. 
Now assume that $G$ is metrizable. Then $\chi(G) =\omega \leq \chi(Z(G))$, as all cardinal functions are assumed to take infinite values.
The reverse inequality $\chi(Z(G))\le \chi(G)$ is trivial.
We have proved that $\chi(Z(G)) = \chi(G)$, so $G$ is \aa\ by Definition \ref{def:chi-abelian}. 
\end{proof}

Another result for $\chi$-minimal groups will be obtained in Theorem \ref{weight:Stoyanov}.

\section{How much commutativity is necessary? Going beyond the class of $\chi$-abelian groups}

In this section we discuss additional conditions on a \mi \ group $G$ that allow us to establish the validity or the failure of 
the equality $t(G)=\chi(G)$ when we do not know whether $G$ is \aa. 

\begin{definition}
In order to measure commutativity of a group $G$, for every positive integer $n$, one defines the \emph{$n$-th center\/} $Z_n(G)$ of $G$ as follows. 
Let $Z_1(G) = Z(G)$. Assuming that $n > 1$ and $Z_{n-1}(G)$ is already defined, consider the canonical projection $\pi \colon G \to G/Z_{n-1}(G)$ and let $Z_n(G) = \pi^{-1} (Z (G/Z_{n-1}(G)))$. This produces an ascending chain of characteristic subgroups $Z_n(G)$ of $G$. 

A group $G$ is {\em nilpotent\/} if $Z_n(G) = G$ for some $n$. In this case, its {\em nilpotency class} $s(G)$ of $G$ is the minimum such $n$. 
\end{definition}

Obviously, the groups of nilpotency class 1 are precisely the abelian groups. 

In order to obtain some analogues of results from Section \ref{sec:4} in the non-commutative case, one of the options is to impose a  ``global'' commutativity condition such as nilpotency or resolvability. The other option is to strengthen  \mi ity.  

 \begin{definition} (\cite{DP})
 A Hausdorff group $G$ is called  {\em totally minimal\/} if all Hausdorff quotients of $G$ are minimal. 
\end{definition}
 
The chain \eqref{minimal:is:locally:minimal} trivially extends, as obviously 
\begin{equation}
\label{minimal:totally:minimal}
\mbox{compact}\to\mbox{totally minimal}\to\mbox{minimal}\to\mbox{locally minimal}.
\end{equation} 
 
The next theorem shows that if one replaces ``minimal'' by ``totally minimal'', then one can extend Corollary \ref{corollary:ct> metrizable}(ii) to the nilpotent case: 

\begin{theorem}
\label{theorem:totally_minimal} Every totally minimal nilpotent group of countable tightness is metrizable. 
\end{theorem}

\begin{proof} We use an inductive argument, using as parameter the nilpotency class $s$ of totally minimal nilpotent group $G$ of countable tightness, and the fact that the center of a minimal group is still a minimal group 
\cite[Chap. 2]{DPS}. 

If $s=1$, then $G$ is abelian, so Corollary \ref{corollary:ct> metrizable}(ii)  applies. 

Suppose that  $s>1$ is an integer and all totally minimal nilpotent groups of nilpotency class $<s$  of countable tightness are metrizable. Let $G$ be a totally minimal group of countable tightness of nilpotency class $s$. Then $G/Z(G)$ is a totally minimal group of countable tightness (as tightness does not increase by taking quotient groups) satisfying $s(G/Z(G)) < s$. Therefore, $G/Z(G)$ is metrizable by our inductive assumption.
Since $Z(G)$ is also metrizable by Corollary \ref{corollary:ct> metrizable}(ii), we deduce that $G$ is metrizable by \cite[Corollary 1.5.21]{AT}.
 \end{proof}

We do not know if  Theorem \ref{theorem:totally_minimal} remains true when we replace ``nilpotent"  by ``resolvable'' (see item (b) of the next example) or  when we replace ``totally \mi''  by ``\mi \ precompact'' (see Question \ref{question:TvsMetr}). 

Item (a) of  the following example, due to Comfort and Grant \cite{CG}, shows that Theorem \ref{theorem:totally_minimal} strongly fails in the non-nilpotent case. Actually, the group in this example is not even resolvable.

\begin{example}\label{example2} Let $\kappa$ be an uncountable cardinal. 

(a)   Let $F$ be a finite simple non-abelian group. Then $G=\Sigma_\kappa(F)$ is totally minimal.
Indeed, every closed normal subgroup of the completion $F^\kappa$ of $G$ has the form $N=F^A$ for some subset $A\subseteq \kappa$. Hence $G$ is totally dense in the compact group $F^\kappa$ (this means that $G\cap N $ is dense in $N$ for every closed normal subgroup of $F^\kappa$). According to a criterion from \cite{DP1}, this implies that $G$ is totally minimal. Combining this with Fact \ref{sigma-product:fact}, we conclude that 
{\em $G$ is a zero-dimensional totally minimal, $\omega$-bounded  (hence, precompact), Fr\' echet-Urysohn group such that $\chi(G)=w(G)=\kappa$.\/}
Replacing $F$ by a connected compact simple Lie group, one obtains a similar example of {\em a non-metrizable Fr\' echet-Urysohn, $\omega$-bounded, totally minimal connected group\/} \cite{DS}. In both cases, $G'=G$ holds; that is, the group $G$ is quite far from being resolvable (in particular, abelian). 

(b) If we replace the finite simple group $F$ in item (a) 
 by any center-free resolvable finite group (e.g., $F=S_3$ or $F = S_4$), then the group 
$G=\Sigma_\kappa(F)$ will also be resolvable and \mi. Hence, {\em $G$ is a zero-dimensional minimal, $\omega$-bounded  (hence, precompact), Fr\' echet-Urysohn resolvable group such that $\chi(G)=w(G)=\kappa$.\/}

Since resolvability of $G$ follows from the  resolvability of $F^\kappa$, it only remains to check that  
 $G$ is minimal. According to Theorem \ref{crit}, it suffices to check that $G$ is essential in $F^\kappa$. Pick a non-trivial (closed) normal subgroup   $N$ of $F^\kappa$. Then there exists $i < \kappa$ such that the image $N_i$ of $N$ under the canonical projection $p_i : F^\kappa \to F$ is non-trivial. Hence, there exists  $x=(x_j)_{j<\kappa}\in N$ such that $x_i= p_i(x) \ne 1$. Since $Z(F)$ is trivial, $x_i\in F$ cannot be central, so there 
 exists $y_i \in F$ such that $u_i := [x_i,y_i]\ne 1$. Pick the element $y=(y_j)_{j<\kappa}\in F^\kappa$ such that $y_j=1$ when $i\ne j\in \kappa$, while  $y_i$ coincides with the previously chosen $y_i \in F$. It is easy to see that 
$$
F_i:=\{f\in F^\kappa: f(j) = 1\mbox{ for all } j<\kappa \mbox{ with }j\ne i\}
$$
is a subgroup of $G$ and $[x,y]\in F_i $ by the choice of $y$. On the other hand, the normality of $N$ in conjunction with $x\in N$ entails  $[x,y]\in N$. Since the $i$th coordinate of $[x,y]$ is equal to $u_i\ne 1$, this proves that $N \cap G \ne \{1\}$.

In Remark \ref{not:tot:minimal:remark}(b) below, we shall show that, unlike the group from item (a), $G$ is not totally minimal. 

(c) More generally, one can take a family $\{F_i: i < \kappa\}$ of finite groups such that all but finite number of them are center-free. 
Then the same argument shows that their $\Sigma$-product $G$ is Fr\' echet-Urysohn, $\omega$-bounded, minimal and non-totally minimal. Moreover, if all $F_i$ are resolvable with a bounded degree of resolvability, then $G$ will be also resolvable. 
\end{example}

\begin{remark}
\label{not:tot:minimal:remark}
(a) The examples above show that Fr\' echet-Urysohn, $\omega$-bounded, minimal resolvable groups need not be metrizable (actually, they may have arbitrarily large character). Yet we do not know if one can trade ``$\omega$-bounded and minimal" for ``totally minimal" in this situation (see  Question \ref{question:TvsMetr}). 

(b)  The minimal resolvable group $G$ built in item (b) of Example \ref{example2} is $\omega$-bounded (hence \cc), but it is not totally \mi. 
  In fact, one has the following general fact: {\em if $G$ is a totally minimal countably compact group, then  $G\cdot \overline{K'} =K$, where $K$ is the completion of $G$}. This follows from the fact that the image $G_1$ of $G$ in the quotient group $K/\overline{K'}$ is totally minimal and \cc, so must be compact, according to \cite{DS}. This yields $G_1=K/\overline{K'}$, and consequently $G\cdot \overline{K'} =K$. 

In our example, $K=F^\kappa$ and $F' \ne F$, as $F$ is resolvable. So $ \overline{K'}=K' = (F')^\kappa$. Therefore, 
$G\cdot \overline{K'} \ne K$. 
\end{remark}

The next theorem shows that commutativity is deeply rooted in Theorem \ref{theorem:A*} (and all its corollaries)  and cannot be  replaced even by the finest possible generalization of abelian.

\begin{theorem}
\label{theorem:LAST0}
For every uncountable cardinal $\kappa$ and for every dense subgroup $H$ of $\T^\kappa$, there exists a minimal  nilpotent  group $G_H$ of nilpotency class 2 having an open subgroup of index $\kappa$ isomorphic to $H \times \T$. Consequently, 
$$
|G_H| = \cont\cdot \kappa\cdot  |H| ,\; \; \;  t(G_H) = t(H\times \T)\; \; \; \mbox{ and }\;\; \; w(G_H)=\chi(G_H)= \chi(H)=\kappa.
$$ 
\end{theorem}

The proof of this theorem will be given in \S \ref{MMM}.

In the sequel we denote by $c(G)$ the connected component of a topological group $G$.

\begin{corollary}
\label{corollary:LAST1}
For every uncountable cardinal $\kappa$, there exists a  Fr\' echet-Urysohn minimal group $G$  of character $\kappa$ such that $c(G)$ is an open normal $\omega$-bounded subgroup (so that $G$ is locally precompact).  
\end{corollary}

\begin{proof} Let $\kappa$ be an uncountable cardinal $\kappa$.
By Fact \ref{sigma-product:fact}, $H=\Sigma_\kappa(\T)$ is a dense $\omega$-bounded, Fr\' echet-Urysohn subgroup of  $\T^\kappa$. Moreover, $H\cong H \times  \T$. Let $G_H$ be the minimal nilpotent group of nilpotency class 2 as in the conclusion of Theorem \ref{theorem:LAST0},
and $L$ be an open subgroup of $G_H$ satisfying $L\cong H\times\T\cong H$. Since $H$ is connected by Fact \ref{sigma-product:fact}(i), so is $L$. Since $L$ is open in $G_H$, we get $L = c(G_H)$.  It remians only to recall that $w(G_H)=\chi(G_H) = \kappa$. 
\end{proof}

\begin{remark} Here is another corollary that can be deduced from the general Theorem \ref{theorem:LAST0}.  Suppose that $\cont < \tau \leq \kappa$ are two cardinals and  $H$ is a dense subgroup of $\T^\kappa$ such that $ t(H) = \tau$. Then $t(H \times \T) = t(H) = \tau$ \cite{Malyhun}, the group $G_H$ given by Theorem \ref{theorem:LAST0} has $t(G_H) = \tau$ and $w(G_H) = \chi(G_H)= \kappa$, i.e.,  one can make the gap between $t(G_H)$ and $\chi(G_H)$  arbitrarily large. However,  $t(G_H)> \cont$ is a blanket condition for this corollary.
 
 Another point is this: unlike Example \ref{example2} one cannot have $G_H$ \cc, as it is never precompact. Also, due to the above mentioned restriction $t(G_H)> \cont$,  this group cannot have countable tightness. Nevertheless, choosing $H$ appropriately (e.g., to sum of all possible subproducts of $\tau$ copies of $\T$) one may still achieve to have an open subgroup of $G_H$ that is $\omega$-bounded (even $\tau$-bounded). 
\end{remark}

\begin{remark} Example \ref{example2} and Corollary \ref{corollary:LAST1} suggest that some degree of ``global'' commutativity may be  necessary for 
the validity of Theorem \ref{corollary:A*} and its corollaries (e.g., Corollary \ref{corollary:t=chi} and Corollary \ref{corollary:ct> metrizable}). 
As a commutativity of global type one can consider nilpotency (or resolvability). This kind of weak commutativity should be compared to the
property of being  \aa \ that we introduced {\em ad hoc}, imposing only the center to be big, regardless of the rest of the group that may be pretty non-commutative, as the example $G = SO_3(\R)\times H$ shows, where $H$ is a non-discrete abelian group. Theorem \ref{theorem:LAST2} shows that  for such a group $G$ the equality $t(G) = \chi(G) = w(G)$ is ensured if the group $H$ is minimal,
while $G$ is quite far from being resolvable (hence, still more from being abelian), since 
$G' = SO_3(\R) \times \{0\}$ is non-trivial and perfect (i.e., $G'' = G'$).
\end{remark}

\section{Large convergent super-sequences in locally minimal groups}

Following \cite{DS2}, call an infinite Hausdorff space $S$ a {\em convergent super-sequence} if $S$ has a single non-isolated point $x$ such that
every neighbourhood of $x$ contains all but finitely many points of $S$.
In such a case we shall say that the super-sequence $S$ {\em converges to } $x$.

Clearly, every non-trivial 
convergent sequence is a convergent super-sequence. 
Obviously, 
\begin{equation}
\label{eq:chi:w}
|S| = \chi(S)= w(S)
\ 
\mbox{ for every convergent super-sequence } S. 
\end{equation}

\begin{definition}
\label{def:s}
For a topological space $X$, we let
$$
\seq (X)= \min\{\kappa: X \mbox{ has no convergent super-sequence of size }\geq \kappa\}.
$$
\end{definition}

It follows from \eqref{eq:chi:w} that 
\begin{equation}
\seq(X) \leq \chi(X) ^+
\ 
\mbox{for every space }X.
\end{equation}
 
For a topological group $G$, the cardinal invariant $\seq(G)$ should be compared with the following one introduced in \cite{DS2}: 
$$
seq (G)= \min\{\kappa: G \mbox{ has a convergent super-sequence $S$ of size } \kappa,  \mbox{ such that } \overline{\langle S \rangle}=G \},
$$
where $\langle S \rangle$ denotes the smallest subgroup of $G$ containing $S$ and the bar denotes the closure in $G$.
Obviously, $seq(G) < \seq(G)$, provided that $seq(G)$ is defined. As shown in \cite{DS2}, $seq(G)^+ <\seq(G)$ may occur (e.g., with $G=\T^\cont$). 

As we saw in Corollary \ref{corollary:ct> metrizable}, a non-metrizable locally minimal locally precompact  abelian group is also non-sequential. Nevertheless, we show that such a group $G$ contains non-trivial converging sequences, actually converging super-sequences of the maximum possible size $\chi(G)$.

\begin{theorem}
\label{super-sequence:in:essential:subgroups}
Every non-discrete, locally essential subgroup $G$ of a locally compact abelian group $K$ contains a super-sequence $S$ converging to $0$ such that $|S|=\chi(K)$. In particular, $\seq(G)= \chi(G) ^+$.
\end{theorem}

The proof of this theorem is given in \S \ref{Sec:10}.

Discrete groups were ruled out in the above theorem as the equality in the conclusion of the theorem fails for (subgroups of) discrete groups. 
Indeed, $\seq(G) = \omega$ for a discrete group $G$, as $G$ contains no super-sequences converging to $0$. On the other hand, 
$\chi(G)=\omega$ for such a group $G$, as all cardinal functions are assumed to take infinite values.  Therefore, $\seq(G)= \chi(G)=\omega$
 for any discrete group $G$. 

\begin{corollary}
\label{super-sequence:in:minimal:groups}
Every non-discrete, locally minimal, locally precompact abelian group $G$ contains a super-sequence $S$ converging to $0$ such that $|S|=\chi(G)$. In particular, $\seq(G)= \chi(G) ^+$.
\end{corollary}

\begin{proof}
Since $G$ is locally precompact, the completion $K$ of $G$ is locally compact.
Since $G$ is locally minimal, it is locally essential in $K$ by Theorem \ref{crit}. Now the conclusion of our corollary follows
from Theorem \ref{super-sequence:in:essential:subgroups} and the equality 
$\chi(G)=\chi(K)$.
\end{proof}

It should be noted that even if the super-sequence $S$ in Theorem \ref{super-sequence:in:essential:subgroups} and Corollary \ref{super-sequence:in:minimal:groups} is as big as possible, it need not topologically generate $G$; that is, $\overline{\langle S \rangle}=G$ may fail. Indeed, there exists an $\omega$-bounded (and totally disconnected) minimal group $G$  such that $\overline{\langle S \rangle}=G$ fails for every convergent super-sequence $S$ in $G$, see \cite[Example 4.2.3]{DTT}. 

\begin{corollary}
\label{cor:super-seq:in:minimal}
Every infinite minimal abelian group $G$ satisfies  $\seq(G) = \chi(G) ^+ =w(G) ^+$.
\end{corollary}

\begin{proof} Note that $G$ is precompact by Theorem \ref{PS}. Being infinite, $G$ cannot be discrete. Now Corollary \ref{super-sequence:in:minimal:groups} applies.
\end{proof}

For non-abelian groups, we have the following result:
    
\begin{theorem}
\label{super-sequence:in:minimal:connected:precompact:groups}
$\seq(G) = \chi(G) ^+ =w(G) ^+$ for every non-trivial, locally minimal, precompact connected group $G$.
\end{theorem}

The proof of this theorem is given in \S \ref{Sec:10}.

\section{The existence of non-trivial convergent sequences in (locally)  minimal groups}

It is clear from the definition that every convergent super-sequence contains a non-trivial convergent sequence.
Therefore, from Corollary \ref{super-sequence:in:minimal:groups}, one gets the following

\begin{corollary}
Every infinite non-discrete locally minimal, locally precompact abelian group has a non-trivial convergent sequence.
\end{corollary}

Similarly, Corollary \ref{cor:super-seq:in:minimal} implies the following

\begin{corollary}\label{corollaryconvergent:sequance}
\cite{Sh2009}
Every infinite minimal abelian group has a non-trivial convergent sequence.
\end{corollary}

  Minimal groups without non-trivial convergent sequences were built by Shakhmatov \cite{Sh2009}. Their underlying group is a free non-abelian group. This example leaves open the question whether some minimal groups, close to being abelian, have non-trivial convergent sequences (see Question \ref{question:convergent:sequance}).

The next theorem establishes the limit of such an extension:

\begin{theorem}\label{Last:Theorem}
There exists a minimal nilpotent group of nilpotency class 2 without non-trivial convergent  sequences
having an open normal countably compact subgroup. 
\end{theorem} 

The proof of this theorem is given in \S \ref{MMM}.

While the group from Theorem \ref{Last:Theorem} is locally countably compact, the 20 years old question of whether every infinite countably compact \mi \ group must have non-trivial convergent sequences 
(\cite[Question 7.2(i)]{D_open_mapping}; repeated as
\cite[Problem 23]{DS-survey} and \cite[Question 9.1(i)]{DM})
remains open.

\begin{remark}  It can be easily derived from Theorem \ref{PS} that every totally minimal nilpotent group is precompact (see \cite{D:Recent:Advances}).  We do not know whether every infinite minimal precompact nilpotent group has a non-trivial convergent sequence (see Question \ref{question:convergent:sequance}). Theorem~\ref{Last:Theorem} shows that this may fail if precompactness is weakened to
local precompactness (actually, even to local countable compactness).   
\end{remark}

Our next theorem shows that one can replace ``abelian'' with ``nilpotent'' in Corollary~\ref{corollaryconvergent:sequance} provided that minimality in its statement is strengthened to total minimality.

\begin{theorem}
\label{totally-minimal:nilpotent}
Every  infinite totally minimal nilpotent group has a non-trivial convergent sequence.  
\end{theorem}

\begin{proof} We are going to use an  inductive argument. Let $s$ denote the nilpotency class of $G$, so that the quotient group $G/Z(G)$ has  nilpotency class $s-1$.  If $s=1$, then the group $G$ is abelian and the assertion follows from Corollary \ref{corollaryconvergent:sequance}. Assume that  $s>1$ and the statement is true for $s-1$. If the center $Z(G)$ is infinite, then it contains a  non-trivial convergent sequence by Corollary \ref{corollaryconvergent:sequance} being a minimal abelian group. Assume that $Z(G)$ is finite. Then $G/Z(G)$ is totally minimal and has  nilpotency class $s-1$, so by the inductive hypothesis there exists a non-trivial convergent sequence $(y_n)$ in $G/Z(G)$. Since $Z(G)$ is finite, the quotient map $q: G \to G/Z(G)$ is a local isomorphism, hence there exists a non-trivial convergent sequence in $G$ as well. 
\end{proof}

Theorem \ref{totally-minimal:nilpotent} gives a positive answer to \cite[Question 7.2(iii)]{D_open_mapping} (repeated 
in \cite[Question 9.1(iii)]{DM}) in the class of nilpotent groups.

\section{Relation between the size and the weight of a locally minimal, precompact abelian group}
\label{sec:8}

For infinite cardinals $\tau$ and $\lambda$, the symbol $\Min(\kappa,\lambda)$ 
denotes the following statement:  There exists a sequence of cardinals $\{\lambda_n: n\in \N\}$ such that 
\begin{equation}\label{min-def}
\lambda=\sup_{n\in\N}\lambda_n\ \  \text{and}\ \  \sup_{n\in\N} 2^{\lambda_n}\leq \kappa \leq 2^{\lambda}.
\end{equation}
We say that the sequence $\{\lambda_n: n\in \N\}$ as above {\em witnesses\/} $\Min(\kappa,\lambda)$. Following \cite{D-GB-S}, 
call a cardinal number $\kappa$ a {\em Stoyanov cardinal}, provided that either $\kappa$ is finite, or $\kappa$ is infinite
and satisfies $\Min(\kappa,\sigma)$ for some infinite cardinal $\sigma$.  

It was proved by Stoyanov \cite{Sto} that $|G|$ is a Stoyanov cardinal for   every minimal abelian group $G$. One of the major result from \cite{D-GB-S}, extending this theorem of Stoyanov, is that $\Min(|G|,w(G))$ holds for every infinite minimal abelian group. We extend this result 
from minimal abelian to locally minimal, $\chi$-abelian groups: 

\begin{theorem}
\label{weight:Stoyanov}
If $G$ is an infinite locally minimal, precompact $\chi$-abelian group, then $\Min(|G|,w(G))$ holds.
\end{theorem}

The proof of this theorem is given in \S \ref{Sec:10}.

Note that ``precompact'' cannot be relaxed to ``locally precompact'' in Theorem \ref{weight:Stoyanov}, since every discrete group is both locally minimal and locally precompact. So every discrete group $G$ such that $|G|$ is not a Stoyanov cardinal will fail to satisfy the conclusion of 
Theorem \ref{weight:Stoyanov}.

It follows from Theorem \ref{weight:Stoyanov} that the cardinalities of the locally minimal precompact free abelian groups coincide with those of the minimal ones. More precisely, we have the following result: 

\begin{corollary}\label{corollary:weight:Stoyanov}
For a free abelian group $F$ the following are equivalent: 
\begin{itemize}
  \item[(a)] $F$ admits a locally minimal, precompact group topology;  
  \item[(b)] $|F|$ is a Stoyanov cardinal; 
  \item[(c)] $F$ admits a minimal group topology. 
\end{itemize}
\end{corollary}
\begin{proof}
Indeed, the implication (a) $\to$ (b) follows from Theorem \ref{weight:Stoyanov}, while the implication (b) $\to$ (c) is due to Stoyanov \cite{Sto}. The remaining  implication (c) $\to$ (a) is trivial.   
\end{proof}

It is important to note that the equivalence between (a) and (c) in the above corollary concerns only {\em free groups}. Indeed, there are many abelian groups admitting 
non-discrete locally minimal precompact group topologies that admit no minimal ones. Here is a series of examples: 

\begin{example}
\begin{itemize}
  \item[(a)] 
 For every positive $n\in \N$,
   the group $\Q^n$, being isomorphic to a subgroup of $\T$, admits a non-discrete,  locally minimal, precompact group topology
by Fact \ref{Last:Fact}. It was proved by Prodanov \cite{P} that $\Q^n$ admits no minimal group topology. 
  \item[(b)] Let $\pi$ be a non-empty set of primes. Then the group $G_\pi : = \bigoplus_{p\in \pi}\Z(p^\infty)$, being isomorphic to a subgroup  of $\T$, admits a  non-discrete,  locally minimal, precompact group topology by Fact \ref{Last:Fact}. It was proved in \cite{DP} that $G_\pi$ admits no minimal group topology if $\pi\ne \Prm$. 
\end{itemize}
\end{example}

The next example is dedicated to the special case of divisible torsion-free groups, which necessarily have the form $G=\Q^{(\kappa)}$
for some cardinal $\kappa$ that obviously coincides with $r(G)$. The case of $r(G) < \omega$ was discussed in item (a) of the above example. 

 \begin{example} 
Consider the group $G = \Q^{(\kappa)}$ for infinite cardinal $\kappa$. It was proved by Prodanov \cite{P} that $G$ 
admits minimal group topologies when $\omega \leq \kappa \leq \cont$. Moreover, for $\kappa > \cont$  the group $G$ 
admits minimal topologies if and only if $\kappa$ is an exponential cardinal \cite{D_div} (and in this case it admits also compact group topologies). 

On the other hand, $G$ admits a precompact locally minimal group topology when $\omega \leq \kappa \leq \cont$ (just take the topology induced on $G$ by any embedding $G \hookrightarrow \T$).  
\end{example}

\begin{example}
Using the construction from Theorem \ref{theorem:LAST0} we give now an example of a minimal nilpotent locally precompact group $G = G_H$ that does not satisfy the conclusion of Theorem \ref{weight:Stoyanov}, i.e., $\Min(|G_H|,w(G_H))$ fails. To this end pick a cardinal $\kappa>\cont$ with $\mbox{cf }\kappa =\omega$. 
 Since $|G_H| = \max\{|H|, \kappa\}$, we shall choose a dense subgroup $H$ of $\T^\kappa$ of size $ \leq  \kappa$. Then $|G_H| = \kappa$. So we have to prove that $\Min(\kappa, \kappa)$ fails. Since $\Min(\kappa, \kappa)$ holds true under GCH,  here only a consistent example can be given.  Under appropriate set-theoretical assumptions we can have $\kappa < 2^{< \kappa}$. Assume that $\kappa=\sup_{n\in\N}\lambda_n$ such that 
 (\ref{min-def})  holds, i.e., $\sup_{n\in\N} 2^{\lambda_n}\leq \kappa$. Our assumption $\kappa < 2^{< \kappa}$ implies that $\kappa < 2^\rho$ for some $\rho < \kappa$.  Pick $n\in\N$ such that  $\rho\leq \lambda_n$, then $\kappa < 2^\rho \leq 2^{\lambda_n} < \kappa$, giving a contradiction. 
 
 One can choose the subgroup $H$ of $\T^\kappa$ to be $\omega$-bounded, so in particular  \cc. Then $G_H$ will be locally \cc, in particular, locally precompact.
 \end{example}

The above example leaves open the question whether Theorem \ref{weight:Stoyanov} remains true for minimal precompact nilpotent groups (see Question \ref{question:Min}). 

\section{Necessary facts}

\begin{lemma}\label{reduction:to:essential0}
Let $\{G_i: i\in I\}$ be an infinite family of topological  groups. If $H$ is a locally essential subgroup of $G = \prod_{i\in I}G_i$, then there exists
a finite subset $J \subseteq I$ such that, considering $G_J =  \prod_{i\in I\setminus J}G_i$ as a subgroup of $G$ in the obvious way, the subgroup $H \cap G_J$ of $G_J$ is essential in $G_J$. 
\end{lemma}

\begin{proof}
Let $V$ be a neighborhood of 1 in $G$ witnessing local essentiality of $H$ in $G$. Then there exists
a finite subset $J \subseteq I$ such that $G_J \subseteq V$. Since $G_J$ is a closed subgroup of $G$, every closed normal subgroup $N$ of $G_J$ is also a closed normal subgroup of $G$. Since it is also contained in $V$, we deduce that $N\cap G_J= N\cap G\ne \{1\}$ whenever $N \ne \{1\}$. 
This proves that $H \cap G_J$ is essential in $G_J$. 
\end{proof}

\begin{corollary}\label{reduction:to:essential}
Let $I$ be an infinite set and let $G$ be a topological group. If $H$ is a locally essential subgroup of $G^I$, then there exists a closed
subgroup $H_1$ of $H$ such that $H_1$ is topologically isomorphic to  an essential subgroup of $G^I$. 
\end{corollary}

\begin{proof} By the above lemma, there exists a finite subset $J \subseteq I$ such that for the closed subgroup  $G_J = G^{I\setminus J}$ of $G^I$, the closed subgroup $H_1= H \cap G_J$ of $H$ is essential in $G_J$. Since $G_J \cong G^I$, we are done. 
\end{proof}

\begin{remark}
\label{remark:(e)}
{\em No dense cyclic subgroup $C$ of $\T^\omega$ is locally minimal\/}. Indeed, suppose that $C$ is locally minimal. Then $C$ must be locally essential in $\T^\omega$ by Theorem \ref{crit}. By Corollary \ref{reduction:to:essential}, $C$ has a subgroup $C_1$ that is isomorphic to an essential subgroup $H$ of $\T^\omega$. Then $H$ is cyclic, as $C_1$ is cyclic. It is easy to see that $\T^\omega$ has two closed non-trivial subgroups $N_1$ and $N_2$ with trivial intersection (e.g., any two distinct coordinate subgroups of $\T^\omega$). Then $H\cap N_1$ and $H\cap N_2$ are non-trivial subgroups of $H$ with trivial intersection. Since $H$ is isomorphic to $\Z$, we get a contradiction.  
\end{remark}

\begin{lemma}\label{claim2} Let $p$ be a prime number and let $K$ be a compact abelian group of exponent $p$. Then every locally essential subgroup of $K$ contains an open subgroup of $K$ topologically isomorphic to  $K$.
\end{lemma} 

\begin{proof} Since $K$ has exponent $p$, $K$ is topologically isomorphic to $\Z(p)^I$ for a set $I$ with $|I|= w(K)$ (\cite{DPS}). Suppose that $H$ is a locally essential subgroup of $K$. Now Lemma \ref{reduction:to:essential0} implies that  there exists a finite subset $J \subseteq I$ such that, considering $K_J = \Z(p)^{I\setminus J}$ as a subgroup of $K$ in the obvious way, the subgroup $H \cap K_J$ of $K_J$ is essential. Since every non-trivial element $x$ of $K_J$ has order $p$, this implies that $x\in H$. Indeed, the subgroup $\langle x\rangle$ of $K_J$ generated by $x$ is isomorphic to $\Z(p)$, so it is closed and has no proper non-trivial subgroups. Therefore, $\langle x\rangle$ is contained in $H \cap K_J\subseteq H$. 
This proves the inclusion $K_J\subseteq H$. Finally, note that $K_J$ is an open subgroup of $K$ topologically isomorphic to $K$.
\end{proof}

\begin{lemma}\label{claim1} If $G$ is a locally essential subgroup of the topological  group $K$, then for every closed central subgroup $N$ of $K$ the subgroup $G\cap N$ of $N$ is  locally essential.
\end{lemma} 

\begin{proof} Let $V$ be a neighborhood of 0 in $K$ witnessing local essentiality of $G$ in $K$. Fix a closed central subgroup $N$ of $K$ and let us check that the neighborhood $V \cap N$ of 0 in $N$ witnesses  local essentiality of $G\cap N$ in $N$. Indeed, let  $L$ be a closed subgroup of $N$ with $L \subseteq V\cap N$ and $L\cap (G\cap N)=\{0\}$. Then $L \cap G=\{0\}$ and $L$ is a closed central (hence, normal)  subgroup of $K$ contained in $V$. Now the local essentiality of $G$ in $K$ with respect to $V$ yields $L=\{0\}$. 
\end{proof}

The proofs of all principal results announced in \S\S  \ref{sec:4}--\ref{sec:8} use essentially the following structure theorem proved in the forthcoming paper \cite{DS1}: 

\begin{theorem}
\label{structure:theorem} Let $K$ be a non-metrizable compact abelian group. Then there exist a closed subgroup $H$ of $K$ and two sequences of cardinals $\{\sigma_p:p\in\P\}$ and $\{\kappa_p: p\in\P\}$ such that
\begin{equation}
\label{eq:1}
H\cong \prod_{p\in\P} \Z(p)^{\sigma_p} \times \prod_{p\in\P} \Z_p^{\kappa_p}
\end{equation}
and $w(H)=w(K)$.
\end{theorem}

\section{Proofs of Theorems \ref{corollary:A*} and  \ref{theorem:B*}}\label{Proofs}

\begin{fact} \cite[Chapter 3]{DPS} For every prime $p$ and every \cag\ $K$ the following conditions are equivalent:
\begin{itemize}
  \item[(a)] $K$ is an inverse limit of finite $p$-groups;
  \item[(b)] $p^nx \to 0$ for every element $x$ of $K$;
  \item[(c)] $K$ admits a structure of a topological $\Z_p$-module. 
\end{itemize}
\end{fact}

The structure of a $\Z_p$-module in item (c) is defined as follows. For $x\in K$ and $\xi\in \Z_p$ consider a sequence of integers $n_k$  in $\Z_p$ such that  $\xi = \lim_k n_k$. Then $(n_kx)$ is a Cauchy sequence in $K$. Let $\xi x = \lim_kn_kx$, this limit exists by the compactness of $K$ and does not depends on the choice of the sequence $(n_k)$ with $\xi = \lim_k n_k$ (due to item (b) of the above fact). 

A compact abelian group satisfying the above equivalent conditions for some prime $p$ is called a {\em pro-$p$-group}.  

For a  pro-$p$-group $K$ and $x\in K$, the closed subgroup generated by $x$ coincides with the cyclic $\Z_p$-submodule $\Z_px = \{\xi x: \xi\in \Z_p\}$ of $K$ generated by $x$. 

\begin{lemma}
\label{uniform:proof}
Let $Z$ be either $\Z_p$ or $\Z(p)$. Suppose that $I$ is an infinite set and $G$ is an essential subgroup of $Z^I$. Then there exist $x=(x_i)\in Z^I$
and $y=(y_i)\in Z^I$ having the following properties:
\begin{itemize}
  \item[(i)] $0\ne x_i\in G\cap Z_i$ for every $i\in I$, where $Z_i$ is the $i$th copy of $Z$ in $Z^I$;
  \item[(ii)] $y_i\not=0$ for every $i\in I$;
  \item[(iii)] $y=(y_i)\in cl_G(\bigoplus_{i\in I} \grp{x_i})$ (note that $\bigoplus_{i\in I} \grp{x_i}\subseteq G$, as (i) holds and $G$ is a subgroup of $Z^I$);
  \item[(iv)] $y\not\in cl_G(A)$ for every set $A\subseteq \bigoplus_{i\in I} \grp{x_i}$ satisfying $|A|<|I|$.
\end{itemize}
In particular, $t(G)\ge |I|$.
\end{lemma}

\begin{proof} Since $G$ is essential in $Z^I$, for every $i\in I$ one can fix a non-zero element $x_i\in G\cap Z_i$. Now item (i) holds.

Let $L$ be the closed subgroup of $Z^I$ generated by $x=(x_i)\in Z^I$. Since $Z$ is a compact $\Z_p$-module,
$L$ coincides with the cyclic $\Z_p$-submodule $\Z_px=\{\eta x:\eta\in \Z_p\}$ of $Z^I$. Since $G$ is essential in $Z^I$, there exists $\eta\in \Z_p$ such that $0\ne y=\eta x \in G$. In particular, $\eta\not=0$. Let $y=(y_i)$. Then $y_i=\eta x_i$ for every $i\in I$.

(ii) Fix $i\in I$. If $Z=\Z_p$, then $y_i=\eta x_i\not=0$, as  both $\eta$ and $x_i$ are non-zero and the ring $\Z_p$ has no divisors of $0$. Suppose now that $Z=\Z(p)$.  Note that the element $\eta\in \Z_p$ cannot be divisible by $p$, since otherwise $\eta = p\zeta$  for some $\zeta\in\Z_p$, so $\eta x= p\zeta x = \zeta (px)= 0$, which contradicts our assumption that $y=\eta x \ne 0$.  Therefore,  $\eta$ is an invertible element of the ring $\Z_p$, so $y_i=\eta x_i\not=0$, as $\eta^{-1} y_i = \eta^{-1}(\eta x_i) = x_i \ne 0$. 

(iii) For $i\in I$, let $M_i$ denote the submodule $\Z_px_i$ of $Z_i$, i.e., the closure of $\hull{x_i}$ in $Z_i$. Then $y\in \prod_{i\in I} M_i$. Since $ \hull{x_i}$ is dense in $M_i$ for each $i\in I$, we conclude that $M=\bigoplus_{i\in I} \grp{x_i}$ is dense in $ \prod_{i\in I} M_i$. 
In particular, $y \in cl_Z(M)$. Since $M\subseteq G$ and $y\in G$, this implies $y\in cl_G(M)$. 

(iv) Since $A\subseteq \bigoplus_{i\in I} \grp{x_i}$, $|A|<|I|$ and $I$ is infinite,  there exists $J\in[I]^{<|I|}$ such that $A\subseteq \bigoplus_{i\in J} \grp{x_i}$. Since $J\subseteq I$ and $|J|<|I|$, we can choose $i_0\in I\setminus J$. Note that $a_{i_0}=0$ for every $a=(a_i)\in A$.
On the other hand, $y_{i_0}\not=0$ by item (ii). This shows that  $y\not\in cl_G(A)$.

The final sentence of the statement of our lemma follows from items (iii) and (iv).
\end{proof}

\begin{corollary}\label{lemma1} 
If $\kappa$ is an infinite cardinal, $p\in\Prm$ and $G$ is a locally essential subgroup of $\Z_p^\kappa$, then $t(G)\ge\kappa$.
\end{corollary}

\begin{proof} 
By Corollary \ref{reduction:to:essential} we can assume without loss of generality that  $G$ is essential in $\Z_p^\kappa$. 
Applying Lemma \ref{uniform:proof} with $I=\kappa$ and $Z=\Z_p$, we get $t(G)\ge|I|=\kappa$. 
\end{proof}

\smallskip 

 \begin{proof}[\bf Proof of Theorem \ref{corollary:A*}] The conclusion obviously holds when $\chi(K)=\omega$ since $t(G)\ge\omega$
by the definition of tightness. So assume from now on that $\chi(K)\ge \omega_1$. Then $K$ is non-metrizable.

 We consider first the case when $K$ is compact. Then $\chi(K) = w(K)$. By Theorem \ref{structure:theorem}, our group $K$ contains, for every $p\in\Prm$, closed subgroups $L_p\cong \Z(p)^{\sigma_p}$ and $M_p\cong \Z_p^{\kappa_p}$, where $\sigma_p$ and $\kappa_p$ are as in the conclusion of Theorem \ref{structure:theorem}. Since $G$ contains a open  subgroup $L_p'$ of $L_p$ with $L_p'\cong L_p$ by Lemma \ref{claim2},  we have 
\begin{equation}
\label{eq:Lp}
t(G)\ge t(L_p')=t(L_p)=t(\Z(p)^{\sigma_p})\ge t(\{0,1\}^{\sigma_p})=\sigma_p
\mbox{ if }
\sigma_p
\mbox{ is infinite}.
\end{equation}
By Lemma \ref{claim1},
$G\cap M_p$ is locally essential in $M_p\cong \Z_p^{\kappa_p}$. Applying Corollary \ref{lemma1}, we conclude that
\begin{equation}
\label{eq:Mp}
t(G)\ge t(G\cap M_p)\ge \kappa_p
\mbox{ if }
\kappa_p
\mbox{ is infinite}.
\end{equation}
From the conclusion of Theorem \ref{structure:theorem} and our assumption we get $w(H)=w(K)\ge \omega_1$. Since 
\begin{equation}
\label{eq:wH}
w(H)=\omega\cdot \sup (\{\sigma_p:p\in\P\}\cup \{\kappa_p: p\in\P\})
\end{equation}
by \eqref{eq:1}, from \eqref{eq:Lp},\eqref{eq:Mp} and \eqref{eq:wH} we get
$t(G)\ge w(H)=w(K)$.

Next, we consider the general case when $K$ is locally compact. There exist $n\in \N$ and a group $K_0$ having an open compact subgroup $C$ such that $ K \cong \R^n \times K_0$.  Obviously, $\chi(C) = \chi(K)\ge \omega_1$. 
Since $G$ is locally essential in $K$, it follows from Lemma
\ref{claim1} that the subgroup $ G \cap C$ of $C$ is locally essential in the compact group $C$. Then by Theorem \ref{theorem:A*},  $\chi(C)=w(C) \leq t(G \cap C) \leq t(G)$. Therefore, $\chi(K) = \chi(C) \leq t(G)$. 
\end{proof}
 
\begin{lemma}\label{lemma7}  Let $p$ be a prime number and let $N$ be a non-trivial abelian pro-$p$-group. 
Then no  locally essential subgroup of the group $K=N^{\omega_1}$ can be radial. 
\end{lemma}

\begin{proof}  By Corollary \ref{reduction:to:essential} we can assume without loss of generality that $G$ is essential. Being a $\Z_p$-module, $N$  contains a closed subgroup $Z$ that is either isomorphic to $\Z_p$, or  isomorphic to $\Z(p)$. Moreover, $G_1= G\cap Z^{\omega_1}$ is essential in $Z^{\omega_1}$. Hence, replacing $G$ by its subgroup $G_1$, we shall assume from now on that $N=\Z_p$ or $N=\Z(p)$. 

Apply Lemma \ref{uniform:proof} with $I=\omega_1$ and $Z=N$ to get elements $x,y\in Z^I=N^{\omega_1}$ as in the conclusion of this lemma.
Define $M=\bigoplus_{i\in I} \grp{x_i}\subseteq G$. Then $y\in cl_G(M)$ by Lemma \ref{uniform:proof} (iii).

Suppose that $G$ is radial. By \cite[Lemma 2.2]{Leek}, there exist a regular cardinal $\gamma$ and a faithfully indexed transfinite sequence $\zeta=(z_\lambda)_{\lambda<\gamma}$ of points of $M$ converging to $y$. Clearly, $\gamma\le|M|=\omega_1$. 

We claim that the equality $\gamma=\omega_1$ holds. Indeed, suppose that $\gamma<\omega_1$. Then $\gamma\le\omega$, as $\gamma$ is a cardinal. Now the set $A=\{z_{\lambda}:\lambda<\gamma\}\subseteq M$ is at most countable, so $y\not\in cl_G(A)$ by item (iv) of Lemma \ref{uniform:proof}. On the other hand, $y\in cl_G(A)$ because $\zeta$ converges to $y$. This contradiction finishes the proof of the equality $\gamma=\omega_1$.

Since the transfinite sequence $\zeta=(z_\lambda)_{\lambda<\omega_1}$ is faithfully indexed, for every countable subset $S$ of $\omega_1$, the set $\{\lambda<\omega_1: supp(z_\lambda)\subseteq S\}$ is at most countable. Therefore, we can use a straightforward transfinite induction to find a cofinal subset $B$ of $\omega_1$ such that the family $\{supp(z_\beta):\beta\in B\}$ is faithfully indexed as well.

Apply the $\Delta$-system lemma \cite{J} to the faithfully indexed family $\{supp(z_{\beta}):\beta\in B\}$ of finite subsets of $\omega_1$ to find an uncountable set $C\subseteq B$ such that $\{supp(z_{\alpha}):\alpha\in C\}$ forms a $\Delta$-system; that is, there exists a finite set $R\subseteq \omega_1$ such that the family $\{supp(z_{\alpha})\setminus R:\alpha\in C\}$ consists of pairwise disjoint sets. Clearly, $C$ is cofinal in $\omega_1$.

Pick any $i\in \omega_1\setminus R$. As $y_i\ne 0$ by item (ii) of Lemma \ref{uniform:proof}, there exists a neighborhood $U_i$ of $y_i$ in $Z_i$ with $0\not \in U_i$. Then $W=G\cap (U_i\times \prod_{j\in \omega_1\setminus \{i\}}Z_j)$ is a neighborhood of $y$ in $G$ such that  $z_\alpha\in W$ only if $i\in supp(z_\alpha)$. Since  $i \not\in R$, this is possible for at most one $\alpha\in C$. Thus, $|\{\alpha\in C: z_\alpha\in W\}|\le 1$.

On the other hand, since $\zeta$ converges to $y$ and the set $C$ is cofinal in $\omega_1$, the cofinal subsequence $(z_\alpha)_{\alpha\in C}$ of $\zeta$ converges to the same limit $y$. Since $W$ is an open neighbourhood of $y$ in $G$, the set $\{\alpha\in C: z_\alpha\in W\}$ must be cofinal in $C$, giving a contradiction.
\end{proof}
 
\begin{proof}[\bf Proof of Theorem \ref{theorem:B*}]
Let $G$ be a locally essential radial subgroup of a \cag\  $K$.
If $K$ is metrizable, then we are done. Suppose that $K$ is non-metrizable. By Theorem \ref{structure:theorem}, there exists $p\in\Prm$ such that (under the notations of this theorem) either $\sigma_p\ge \omega_1$ or $\kappa_p\ge \omega_1$. Therefore,  $K$ contains a closed subgroup $H$ such that either $H\cong \Z(p)^{\omega_1}$
or $H\cong \Z_p^{\omega_1}$. By Lemma \ref{claim1}, $G\cap H$ is a  locally essential subgroup of $H$. 
Hence we can apply Lemma \ref{lemma7} to claim that $G \cap H$ is not radial. Consequently, $G$ is not radial either, a contradiction.  \end{proof}

\section{Proof of Theorems \ref{theorem:LAST0} and \ref{Last:Theorem}}\label{MMM}

\begin{proof}[\bf Proof of Theorem \ref{theorem:LAST0}]
Let $\kappa$ be an uncountable cardinal, let $K= \T^\kappa$ and $\widehat K = \Z^{(\kappa)}$. Denote by 
${\mathcal H}(K)$ 
the Heisenberg group introduced by Megrelishvili in \cite{Meg} (see also  \cite{DMeg}) defined as follows. The topological group $\mathcal H(K)$ has as a supporting topological space  the Cartesian product $\T \times K \times \widehat K$ equipped with the product topology, and group operation defined for 
\begin{equation}\label{(Last)}
u_1=(a_1,x_1,f_1), \hskip 0.4cm u_2=(a_2,x_2,f_2)
\end{equation}
in $\mathcal H(K)$ by 
$$
u_1 \cdot u_2 = (a_1+a_2+f_1(x_2), x_1+x_2, f_1 +f_2)
$$ 
In other words,  $\mathcal H(K)$ is the semidirect product  $(\T \times K ) \leftthreetimes \widehat K$ of the groups $\widehat K$ and  $\T \times K$, where the action of $\widehat K $ on $\T \times K$ is defined by $(f,(a,x))\mapsto (a+f(x),x)$ for $f\in \widehat K $ and $(a,x) \in \T \times K $. Then $\mathcal H(K)$ is a minimal and locally compact \cite{Meg}. 

It is convenient to describe the group $\mathcal H(K)$ in the matrix form
$$
\left(\begin{matrix}
  1 &  \widehat K & \T \cr
\,0 & 1 & K \cr
\,0 & 0 & 1 \cr
                 \end{matrix}
                 \right)
$$
so that the multiplication of two elements of $\mathcal H(K)$ is carried out precisely by the rows-by-columns rule for multiplication of $3 \times 3$ matrices, where the ``product'' of an element $f\in \widehat K$ and and element $x\in K$ has value $f(x) \in \T$. For a subgroup $H$ of $K$ we denote by $G_H$ the set 
$$
 \left(\begin{matrix}  
  1 & \widehat K  & \T \cr
\,0 & 1 &  H  \cr
\,0 & 0 & 1 \cr
                 \end{matrix}
                 \right) :=  \left\{\left(\begin{matrix}
  1 & h & t \cr
\,0 & 1 & x \cr
\,0 & 0 & 1 \cr
                 \end{matrix}
                 \right)\in \mathcal H(K): h\in \widehat K, x\in H, t\in \T
                                  \right\}. 
$$
 It is easy to check that $G_H$ is a subgroup of $\mathcal H(K)$. It is dense in $\mathcal H(K)$ iff $H$ is dense in $K$. We identify $K$ with $\{0_\T \} \times K \times \{0_{ \widehat K} \}$,  $ \widehat K$ with $\{0_\T\} \times \{0_K\} \times  \widehat K$
 and  $ \T$ with $\T \times \{0_K\} \times  \{0_{\widehat K}\}$. 

Elementary computations for the commutator $[u_1,u_2]$ of $u_1, u_2\in \mathcal H(K)$ as in (\ref{(Last)})
 give $[u_1,u_2] = u_1u_2u_1^{-1}u_2^{-1}= (f_1(x_2)-f_2(x_1),0_E,0_F)$ hence 
$$
Z(\mathcal H(K))=\left(\begin{matrix}
   1 & 0 & \T \cr
 \,0 & 1 & 0 \cr
 \,0 & 0 & 1 \cr
                 \end{matrix}
                 \right).
$$
The quotient $\mathcal H(K) /Z(\mathcal H(K))$ is isomorphic to the (abelian) direct product $K \times  \widehat K$. Therefore,  $\mathcal H(K)$ is a nilpotent group of class 2. 

To prove that $G_H$ is minimal, it suffices to check that every non-trivial normal subgroup $N$ of $\mathcal H(K)$ non-trivially meets $Z(\mathcal H(K))\subseteq G_H$. Let $u \in N$, $u\ne 1$. If $u\in  Z(\mathcal H(K))$ we are done. If $u\not \in Z(\mathcal H(K))$, then there exists $v\in \mathcal H(K)$ such that $[u,v]\ne 1$.  Since $vuv^{-1}\in N$, we conclude that $1\ne [u,v]=uvu^{-1}v^{-1}\in N$. On the other hand, $[u,v]\in Z(\mathcal H(K))$ as we have seen above. This proves that $N \cap Z(\mathcal H(K))\ne \{1\}$. This establishes the  minimality of the group $G_H$. 

Obviously,
$$
O =  \left(\begin{matrix}  
  1 & 0 & \T \cr
\,0 & 1 &  H  \cr
\,0 & 0 & 1 \cr
                 \end{matrix} \right) 
$$ 
is an open subgroup of $G_H$ and $O \cong H \times \T$. Moreover, $O$ contains  $Z(\mathcal H(K))=\mathcal H(K)'$. Hence, $O$ is a normal subgroup of $\mathcal H(K)$.  
\end{proof}

\begin{proof}[\bf Proof of Theorem \ref{Last:Theorem}]
This proof uses the same construction with only difference consisting in replacement of the group $\T$
by the group $\Z(2)$. To achieve this, we first take the group $G_H$ built in the above proof with ingredients $K = \Z(2)^\cont$ and the dense subgroup $H$ of $K$ to be  countably compact and without non-trivial convergent  sequences. 
Such a group $H$ was recently built Hru\v s\' ak, van Mill, Ramos-Garc\' \i a and Shelah \cite{H-S}.  Then 
$$
L:=  \left(\begin{matrix}  
  1 & \widehat K  & \Z(2) \cr
\,0 & 1 &  H  \cr
\,0 & 0 & 1 \cr
                 \end{matrix}
                 \right) =  \left\{\left(\begin{matrix}
  1 & h & t \cr
\,0 & 1 & x \cr
\,0 & 0 & 1 \cr
                 \end{matrix}
                 \right)\in \mathcal H(K): h\in \widehat K, x\in H, t\in \Z(2)
                                  \right\}
$$
is a subgroup of $G_H$. For brevity, we identify $H$ and $\Z(2) \oplus H$ with subgroups
$$
\left(\begin{matrix}  
  1 & 0  & 0 \cr
\,0 & 1 &  H  \cr
\,0 & 0 & 1 \cr
                 \end{matrix}
                 \right) 
         \ \        \mbox{ and }\ \ 
\left(\begin{matrix}  
  1 & 0  & \Z(2) \cr
\,0 & 1 &  H  \cr
\,0 & 0 & 1 \cr
                 \end{matrix}
                 \right)
$$
of $L$, respectively. According to \cite[Lemma 5.16 (C)]{DM},  the following holds:
\bit
\item [(i)] \ \ $L$ is a locally precompact, minimal nilpotent group (of class 2) of weight $\cont$;
\item [(ii)] \ \ $\Z(2) \oplus H$ is an open normal subgroup of $L$;
\item [(iii)] \ \ $L$ has finite  exponent $4$;
\item [(iv)] \ \ $L$ has no convergent sequences, as $H$ has no convergent  sequences;
\item [(v)] \ \ $L$ is locally countably compact, as $H$ is countably compact.
\eit
Therefore, $L$ is the desired example.
\end{proof}

\begin{remark}
The countably compact Boolean group $H$ without non-trivial convergent sequences built in \cite{H-S} and used in the proof of Theorem
\ref{Last:Theorem} above cannot be minimal. The reason for this is the fact that $H$ is abelian, so $H$ must have non-trivial convergent sequences by Corollary \ref{corollaryconvergent:sequance}.

An alternative way to prove that $H$ is not minimal is to notice that {\em a minimal  Boolean group $H$ is compact\/}, hence $H$ must have non-trivial convergent sequences. Indeed, since $H$ is minimal and abelian, its completion $K$ is compact by Theorem \ref{PS}.
Since $H$ is minimal, it must be essential in $K$ by  Theorem \ref{crit}. Since $H$ is Boolean, this yields $H=K$.
\end{remark}

\section{Proofs of Theorems \ref{super-sequence:in:essential:subgroups}, 
\ref{super-sequence:in:minimal:connected:precompact:groups} and \ref{weight:Stoyanov}}
\label{Sec:10}

\begin{lemma}
\label{lemma:0.1} Let $I$ be an infinite set, $\{K_i: i\in I\}$ a family of non-trivial topological groups and $G$  a  locally  essential subgroup of the product $K=\prod_{i\in I} K_i$. Then $G$ contains a super-sequence $S$ converging to $1$ such that $|S|=|I|$.
\end{lemma}

\begin{proof}  We show first that we can assume that $G$ is an essential subgroup of $K$ without loss of generality. 
According to  Lemma \ref{reduction:to:essential0} there exists a finite subset $J \subseteq I$ such that considering $K_J =  \prod_{i\in I\setminus J}K_i$ as a subgroup of $K$ in the obvious way, the subgroup $G_J = G \cap K_J$ of $K_J$ is essential. As $|J| = |I|$, we can replace $K$ by $K_J$ and $ G$ by $G_J$ and assume in the sequel that $G$ is an essential subgroup of $K$.   

 For each $i\in I$, use essentiality of $G$ in $K$ to fix $g_i\in G\cap K_{i}$ with  $g_{i}\not=1$.   Since   $K_{i}\cap K_{j}=\{1\}$ 
    whenever $i,j\in I$ and $i\not=j$, it follows that  $\{g_i:i\in I\}$ is a a faithfully indexed family of elements of $G$. In particular, for $S=\{0\}\cup \{g_i:i\in I\}$ we have $|S|=|I|$. Furthermore, $S$ is a super-sequence converging to $1$. 
Clearly, $S\subseteq G$.
\end{proof}

\begin{proof}[\bf Proof of  Theorem \ref{super-sequence:in:essential:subgroups}] 
We have to prove that  $G$ contains a super-sequence $S$ converging to 0 with $|S| =\chi(K)$.  This is trivially true if $K$ is metrizable, since then $G$ is metrizable and non-discrete, so contains an infinite sequence converging to $0$. 

Assume now that $K$ is non-metrizable. By a well-known structure theorem of the locally compact abelian groups \cite{DPS}, we can assume, up to isomorphsim, that $K= \R^n \times L$, where $n \in \N$ and the group $L$ contains an open compact subgroup $C$. Since $K$ is non-metrizable, 
$\chi(K) = \chi(C) = w(C)> \omega$. Apply Theorem  \ref{structure:theorem} to $C$ to find a   closed subgroup $H$ of $C$ (hence of $K$ as well)
which is a direct product of $w(K)$-many non-trivial second countable groups.  From the local essentiality of $G$ in $K$ and Lemma \ref{claim1}, we deduce that $G\cap H$ is locally essential in $H$.  By Lemma \ref{lemma:0.1}, $G\cap H$ contains a super-sequence $S$ converging to $0$ such that $|S|=w(C)=\chi(K)$.
\end{proof}

\begin{proof}[{\bf Proof of Theorem \ref{super-sequence:in:minimal:connected:precompact:groups}}]
Since $G$ is precompact and connected,  the completion 
$K$ of $G$ is a compact connected group.
Then  
$K=K' Z(K)$
\cite{HM}. Moreover, it follows from a theorem of Varopoulos \cite{HM} that 
$$
K/Z(K)\cong  K'/(K'\cap Z(K)) = K'/Z(K') \cong \prod_{i\in I}L_i,
$$
 where each $L_i$ is a simple compact Lie group. Moreover, if for each $i\in I$ one denotes by $\widetilde L_i$ the covering group of $L_i$, then $Z_i=Z(\widetilde L_i)$ is finite, and  for some closed central subgroup $N$ of $L=\prod_{i\in I} \widetilde L_i$ one has $K'\cong L/N$.
Let $q: L \to K'$ be the canonical map, provided $K'$ is identified with the quotient $L/N$ under the above mentioned isomorphism. Since $F_i=N \cap \widetilde L_i$ is finite for every $i\in I$, the compact subgroup $L_i^*=q(\widetilde L_i)$ of $K$ is non-trivial,  as $q(\widetilde L_i) = \{1\}$ would imply $L_i\leq N = \ker q$, which contradicts the fact that $L_i$ is a non-trivial connected group while $N$ is totally disconnected, being a subgroup of $Z(L) = \prod_{i\in I}Z_i$. Moreover,  $L_i^*$  is normal, as $\widetilde L_i$ is normal in $L$ and $q$ is surjective. 

Since $G$ is locally minimal, $G$
is
 locally essential in $K$ by Theorem \ref{crit}. Fix an open neighborhood $V$ of $1$ witnessing local essentiality of $G$ in $K$. 
Since $q^{-1}(V)$ is a neighborhood of $1$ in $L$, there exists a co-finite subset $J$ of $I$ such that $\prod_{i\in J} \widetilde L_i \subseteq q^{-1}(V)$ (here we identify $\prod_{i\in J} \widetilde L_i$ with a subgroup of $L$ in the obvious way). In particular, $L_i^*\subseteq V$ for every $i\in J$. By our choice of $V$, it follows that $G\cap L_i^*$ is non-trivial  for every $i\in J$.  Fix an element $x_i^*\in G\cap L_i^*$ with $x_i^*\ne e$ and choose an element $x_i \in \widetilde L_i$ such that $q(x_i) = x_i^*$. It is easy to prove that the set $S=(x_i)_{i\in J}$ is a super-sequence in $L$ converging to $1$. Then  $S_l:=q(S) = (x_i^*)_{i\in J}$  is a super-sequence in $G$ converging to $1$.
Obviously, $|S_l| = |J|=w(K/Z(K))$. 

Note that $G \cap Z(K) \subseteq Z(G)$.  Since $\overline{G}$ coincides with $K$,  we have $ Z(G) \subseteq Z(K)$, so  $G \cap Z(K) = Z(G)$. Since every subgroup of  $Z(K)$ is normal, we conclude that $Z(G)$ is a locally essential subgroup of $Z(K)$. Now consider two cases. 

If $Z(G)$ is infinite, Theorem \ref{super-sequence:in:essential:subgroups} applies to produce a super-sequence $S_z$ in $Z(G)$ converging to $1$ such that $|S_z|=w(Z(K))$. 

Let $S = S_z \cup S_l$. Then $S$ is a super-sequence in $G$ converging to $1$. Moreover,  
$$
|S| = \max\{w(K/Z(K)), w(Z(K))\}= w(K) = w(G).
$$

Now assume that $Z(G)$ is finite. Since $Z(G)$ is locally essential in $Z(K)$,
there exists a \nbd \  $V$ of 
$1$
 in $Z(K)$ witnessing that. Since $Z(G)$ is finite, we can choose  a neighbourhood $U$  of  
$1$
such that 
$\overline{U} \cap Z(G) = \{1\}$.
 Then for every subgroup $N$ of $Z(K)$ contained in $U\cap V$, the  closure of $N$ is contained in $\overline{U}$, hence misses $Z(G)$, by the choice of $U$. Therefore,  the definition of local essentiality of $Z(G)$ in $Z(K)$ entail that $U\cap V$   
contains no non-trivial subgroups of $Z(K)$, i.e., $Z(K)$ is a compact abelian NSS group. It is well known that this implies that $Z(K)$ is a Lie group; in particular, $Z(K)$ is metrizable. This yields  
$$
w(K) =\max\{w(K/Z(K)), w(Z(K))\}= w(K/Z(K)) =  |S_l|.
$$
 Hence, the super-sequence $S_l$ converging to $1$ in $G$ is the desired one. 
\end{proof}

\begin{lemma}
\label{size:of:essential:subgroups:of:powers} Let $H$ be any non-trivial abelian topological
group and $\sigma$ an infinite cardinal.  Then every locally essential subgroup of $H^\sigma$ has size at least $2^\sigma$.
\end{lemma}

\begin{proof} Let $G$ be a  locally  essential subgroup of $H^\sigma$.  By Corollary \ref{reduction:to:essential} we can assume without loss of generality that $G$ is essential. Since $H$ is non-trivial, we can choose $z\in H\setminus \{0\}$. For every  $f\in2^\sigma$ define $h_f \in H^\sigma$ by $h_f(\alpha)=z$ when  $f(\alpha)=1$ and $h_f(\alpha)=0$ otherwise. Let $N_f$ be the smallest closed subgroup of $H^\sigma$ containing $h_f$.  Since $N_f$ is a non-trivial closed subgroup of $H^\sigma$ and $G$ is essential in $H^\sigma$, we can choose $x_f\in (G\cap N_f)\setminus \{0\}$. We claim that the map $f\mapsto x_f$ is an injection of $2^\sigma$ into $G$, which gives $|G|\ge 2^\sigma$. Indeed, assume $f_0,f_1\in 2^\sigma$ and $f_0\not=f_1$. Then there exist  $i\in\{0,1\}$ and $\alpha<\sigma$ such that $f_i(\alpha)=0$ and $f_{1-i}(\alpha)=1$. Now $x_{f_i}(\alpha)=0$ while $x_{f_{1-i}}(\alpha)=z\not=0$, which yields $x_{f_0}\not=x_{f_1}$.
\end{proof}

\begin{proof}[\bf Proof of Theorem \ref{weight:Stoyanov}]  
The completion $K$ of $G$ is a compact group.  We consider first the case when $G$ is abelian. Then $K$ is abelian as well, $w(K)=w(G)$ and 
$G$ is  locally  essential in $K$ by Theorem \ref{crit}.

If $w(K)=\omega$, then $\omega\le |G|\le 2^{w(G)}=2^\omega$, and now the sequence $\{n:n\in\N\}$ witnesses $\Min(|G|,\omega)$.

Assume now that $w(K)>\omega$. Then $K$ is non-metrizable, and we can apply Theorem  \ref{structure:theorem} to $K$.  Define $P_\sigma=\{p\in\Prm: \sigma_p\ge \omega\}$ and  $P_\kappa=\{p\in\Prm: \kappa_p\ge \omega\}$. Since $w(H)=w(K)\ge\omega_1$, from \eqref{eq:1} it follows that 
\begin{equation}
\label{eq:0.0}
w(H)= \sup(\{\sigma_p:{p\in P_\sigma}\}\cup \{\kappa_p:{p\in P_\kappa}\}). 
\end{equation}

(i) Let $p\in P_\sigma$. By Lemmas \ref{claim1} and \ref{claim2}, $G$ contains a copy of $\Z(p)^{\sigma_p}$, so $|G|\ge |\Z(p)^{\sigma_p}|=2^{\sigma_p}$.

(ii) Let $p\in P_\kappa$. Let $L_p$ be a closed subgroup of $K$ with $L_p\cong \Z_p^{\kappa^p}$. By Lemma \ref{claim1}, $G\cap L_p$ is  locally essential in $L_p$, so from Lemma 
\ref{size:of:essential:subgroups:of:powers} we get  $|G|\ge |G\cap L_p|=2^{\kappa_p}$.

From (i) and (ii) it follows that 
\begin{equation}
\label{eq:0.1}
|G|\ge \sup\left(\left\{2^{\sigma_p}:{p\in P_\sigma}\right\}\cup \left\{2^{\kappa_p}:{p\in P_\kappa}\right\}\right).
\end{equation}
Let $\{p_{2m}:m\in\N\}$ be an enumeration of the set $P_\sigma$ and  $\{p_{2m+1}:m\in\N\}$ be an enumeration of the set $P_\kappa$ (repetitions are allowed). Define $\lambda_n=\sigma_{p_n}$ for an even $n$ and $\lambda_n=\kappa_{p_n}$ for an odd $n$. Since $w(H)=w(K)=w(G)\le 2^{w(G)}$, from  \eqref{eq:0.0} and \eqref{eq:0.1} we conclude that the sequence $\{\lambda_n:n\in\N\}$ witnesses $\Min(|G|,w(G))$. 

In the general case,  $w(G)=\chi(G) = \chi(Z(G)) =w(Z(G))$,
as $G$ (and consequently, also $Z(G)$) is precompact. Since the abelian group $Z(G)$ is locally minimal \cite{LocMin}, we can apply the above argument and conclude that $\Min(|Z(G)|,w(Z(G)))$ holds.  Hence, there exists a sequence of cardinals $\{\lambda_n:{n\in \N}\}$ such that 
$w(Z(G))= w(G)= \sup(\{\lambda_n:{n\in \N}\}$ and $\sup_n 2^{\lambda_n}\leq |Z(G)| \leq 2^{w(Z(G))}=2^{w(G)}$. 
This  obviously gives  $\sup_n 2^{\lambda_n}\leq |Z(G)| \leq |G| \leq 2^{w(G)}$. Hence, $\Min(|G|,w(G))$ holds.
\end{proof}

\section{Open problems}

\begin{question}
\label{question:0}
Does the tightness of a locally minimal abelian group coincide with its character?
\end{question}

\begin{question}
\label{question:1}
Is every locally minimal (abelian) group of countable tightness necessarily metrizable? Is it necessarily Fr\' echet-Urysohn or sequential? 
\end{question}

\begin{question}\label{question:TvsMetr}
\begin{itemize}
  \item[(a)]  Is every totally minimal (precompact) resolvable group of countable tightness metrizable?
    \item[(b)] Is every minimal precompact nilpotent group of countable tightness  metrizable?
\end{itemize}
\end{question}

\begin{question}
\label{question:2}
Is every pseudoradial minimal abelian group necessarily metrizable?
\end{question}

Let $\tau$ be a cardinal. According to \cite{Arh0}, a transfinite sequence $\{x_\alpha:\alpha<\tau\}$ of points of topological space $X$ is called a {\em free sequence of length $\tau$\/} provided that
$\overline{\{x_\alpha:\alpha<\beta\}}\cap \overline{\{x_\alpha:\alpha\ge \beta\}}=\emptyset$ for every $\beta<\tau$.
The supremum of lengths of all free sequences in a topological space $X$ shall be denoted by $fs(X)$. 

Arhangel'ski\u\i\ proved in \cite{Arh0} that $t(X)=fs(X)$ for every compact space $X$. While the inequality $fs(X)\le t(X)$ holds for an arbitrary space $X$ \cite{Arh0},  the converse inequality $t(X)\le fs(X)$ need not hold even for a $\sigma$-compact space $X$ \cite{Okunev}. 

Since minimal abelian groups are precompact by Theorem~\ref{PS}
(so have some form of compactness), one may wonder whether Arhangel'ski\u\i\'s characterization of the tightness of compact spaces in terms of cardinalities of free sequences holds also for minimal abelian groups.

\begin{question}
\label{question:3}
Is it true that  $t(G)=fs(G)$ for every minimal abelian group $G$?
\end{question}

Our next question is related to Corollary \ref{corollaryconvergent:sequance}. In Theorem \ref{Last:Theorem} we exhibited 
 a nilpotent locally precompact (actually, locally contably compact) minimal group without non-trivial convergent sequences. This leaves open the following question: 
  
\begin{question} \label{question:convergent:sequance}
\begin{itemize}
   \item[(a)] Does every infinite minimal resolvable precompact group admit a non-trivial convergent sequence? What about  infinite minimal precompact nilpotent groups? 
   \item[(b)] Is the answer of the questions in item (a) positive if the groups in question are also countably compact? 
\end{itemize}
\end{question}

\begin{question}\label{question:Min}
Does $\Min(|G|,w(G))$ hold true for every minimal precompact nilpotent group $G$? 
\end{question}

{\sc Acknowledgements.} The authors are indebted to the referee for her/his careful reading and useful suggestions.

\end{document}